\documentclass[11pt]{article}

\usepackage{amsmath,amssymb,amsthm,amscd,epsfig}
\newcommand{\vp}{\varepsilon}

\newcommand{\bb}[1]{\mathbb{#1}}

\newcommand{\cl}[1]{\mathcal{#1}}

\theoremstyle{plain}

\newtheorem{thm}{Theorem}

\newtheorem{cor}{Corollary}

\theoremstyle{definition}

\newtheorem{defn}{Definition}

\theoremstyle{remark}

\begin{document}

\title{Overlap functions for measures in conformal iterated function systems}

\author{Eugen Mihailescu and Mariusz Urba\'nski}

\date{}
\maketitle

\begin{abstract}
We study conformal iterated function systems (IFS) $\mathcal S = \{\phi_i\}_{i \in I}$ with arbitrary overlaps, and measures $\mu$ on limit sets $\Lambda$, which are projections of equilibrium measures $\hat \mu$ with respect to a certain lift map $\Phi$ on $\Sigma_I^+ \times \Lambda$. No type of Open Set Condition is assumed. We introduce a notion of overlap function and overlap number for such a measure $\hat \mu$ with respect to $\cl S$; and, in particular a notion of (topological) overlap number $o(\cl S)$. These notions take in consideration the $n$-chains between points in the limit set.
We prove that $o(\cl S, \hat \mu)$ is related to a conditional entropy of $\hat \mu$ with respect to the lift $\Phi$. Various types of projections to $\Lambda$ of invariant measures are studied. We obtain upper estimates for the Hausdorff dimension $HD(\mu)$ of  $\mu$ on $\Lambda$, by using pressure functions and $o(\cl S, \hat \mu)$. In particular, this applies to projections of Bernoulli measures on $\Sigma_I^+$. 
Next, we apply the results to Bernoulli convolutions $\nu_\lambda$  for $\lambda \in (\frac 12, 1)$, which correspond to self-similar measures determined by composing, with equal probabilities, the contractions of an IFS with overlaps $\cl S_\lambda$. We prove that for \textit{all} $\lambda \in (\frac 12, 1)$, there exists a relation between  $HD(\nu_\lambda)$ and the overlap number $o(\cl S_\lambda)$. The number $o(\cl S_\lambda)$ is approximated with integrals on $\Sigma_2^+$ with respect to the  uniform Bernoulli measure $\nu_{(\frac 12, \frac 12)}$. We also estimate $o(\cl S_\lambda)$ for certain values of $\lambda$.
\end{abstract}



\section{Introduction and outline.}

Iterated function systems (IFS) have been studied by many authors, and a lot about their theory is known. In many instances, systems which satisfy the Open Set Condition were studied. 
When arbitrary overlaps of the images of the contractions  are allowed, the theory is different and the results from the case of Open Set Condition do not work anymore.

Let us consider a finite set $I$ and an iterated function system $\cl S = \{\phi_i, i \in I\}$ consisting of injective conformal contractions $\phi_i$ defined on the closure of an open set $V \subset \bb R^q, q \ge 1$. 
Denote by $\Sigma_I^+$ the one-sided space $\{\omega = (\omega_1, \omega_2, \ldots), \omega_j \in I, j \ge 1\}$, with its shift endomorphism $\sigma: \Sigma_I^+ \to \Sigma_I^+, \sigma(\omega) = (\omega_2, \omega_3, \ldots)$. For an arbitrary sequence $\omega$ and for an integer $n \ge 1$, let the $n$-truncation $\omega|_n$ be the finite sequence $(\omega_1, \ldots, \omega_n)$. Also by $[i_1\ldots i_n]$ we denote the $n$-cylinder $\{\omega \in \Sigma_I^+, \ \omega_1=i_1, \ldots, \omega_n = i_n\}, \ n \ge 1, i_1, \ldots, i_n \in I$.

Let denote now by $\Lambda$ the fractal \textit{limit set} of the iterated function system $\cl S$, where: $$\Lambda:= \mathop{\cup}\limits_{\omega \in \Sigma_I^+} \mathop{\cap}\limits_{n \ge 1} \phi_{\omega|_n}(V)$$
Since all the maps $\phi_i$ are contractions, we can define the canonical coding map $\pi: \Sigma_I^+ \to \Lambda, \ \pi(\omega) = \mathop{\lim}\limits_{n \to \infty} \phi_{\omega_1}\circ \phi_{\omega_2}\circ\ldots\circ\phi_{\omega_n}(V)$, for all $\omega = (\omega_1, \omega_2, \ldots) \in \Sigma_I^+$. The singleton $\pi(\omega)$ will also be denoted by $\phi_{\omega_1}\circ \phi_{\omega_2}\circ \ldots$, as this infinite composition is in fact a point. We will denote the composition $\phi_{i_1}\circ \ldots \circ \phi_{i_m}$ also by $\phi_{i_1\ldots i_m}$, for $m \ge 1, i_j \in I, 1\le j \le m$. The map $\pi$ is  called the canonical projection onto the limit set $\Lambda$ of the system  $\cl S$.
Various properties of IFS's with overlaps were studied by several authors, for eg in \cite{F}, \cite{So}, \cite{H}, \cite{PSS}, \cite{PS}, \cite{MU-PAMS11}, etc.  \
Let us fix some more terminology and notation.

\begin{defn}\label{ov}
By \textit{overlaps} we mean intersections of type $\phi_i( \Lambda) \cap \phi_j(\Lambda) \ne \emptyset, \ i \ne j$. 
If for a point $x \in \Lambda$ and an integer $m \ge 1$,  there exists a point $\zeta \in \Lambda$ and a finite sequence $i_1, \ldots i_m \in I$ such that $\phi_{i_1}\circ \ldots \circ \phi_{i_m}(\zeta) = x$, then $\zeta$ is called an \textit{m-root} of $x$, and $(i_1, \ldots, i_m)$ is called an \textit{m-chain} from $\zeta$ to $x$.
\end{defn}

In general, the number of roots/overlaps depends on the point $x \in \Lambda$, so it is not constant. Notice also that the $m$-chain from a certain root $\zeta$ to $x$ is not uniquely defined, i.e there may exist two different $m$-chains $(i_1, \ldots, i_m)$ and $(j_1, \ldots, j_m)$ so that $\phi_{i_1\ldots i_m}(\zeta) = \phi_{j_1\ldots j_m}(\zeta) = x$.
Considering the above,  how can we define a good notion of average number of overlaps of the IFS $\cl S$, and how is such a notion dependent on a probability measure $\mu$ on $\Lambda$; also, how does such a number of overlaps affect the Hausdorff dimension of $\mu$?
It is clear that we have to look at $n$-roots of points, since the limit set $\Lambda$ is invariant under the system $\cl S$, i.e $\Lambda = \mathop{\cup}\limits_{i \in I} \phi_i(\Lambda)$, thus for $k$-iterations of $\cl S$ we have $\Lambda = \mathop{\cup}\limits_{i_1, \ldots, i_k \in I} \phi_{i_1\ldots i_k}(\Lambda)$, for any $k \ge 2$. This hints to the fact that the  overlap number should be given by an average rate of growth of the number of $n$-chains between points in the limit set.
Another question is, what probabilities $\mu$ on $\Lambda$ should be considered, and what roots in $\Lambda$ do we use. Some $n$-roots and $n$-chains which are non-generic with respect to $\mu$ and to a lift map $\Phi: \Sigma_I^+ \times \Lambda \to \Sigma_I^+ \times \Lambda$ will thus be ignored when defining the overlap number relative to $\mu$. 

Besides the canonical coding projection $\pi: \Sigma_I^+ \to \Lambda$, one can consider also the projection $\pi_2: \Sigma_I^+ \times \Lambda \to \Lambda, \ \pi_2(\omega, x) = x$, and the projection $\tilde \pi: \Sigma_I^+ \times \Sigma_I^+ \to \Sigma_I^+ \times \Lambda, \ \tilde \pi(\omega, \eta) = (\omega, \pi\eta)$; so we obtain projections of $\sigma$-invariant measures on $\Sigma_I^+$, $\Phi$-invariant measures on $\Sigma_I^+ \times \Lambda$ or $\tilde \Phi$-invariant measures on $\Sigma_I^+ \times \Sigma_I^+$ (where $\tilde \Phi$ is a lift of $\Phi$ to $\Sigma_I^+ \times \Sigma_I^+$).  In \textbf{Theorem \ref{equal}} we will prove that, for Bernoulli measures, the corresponding projection measures on $\Lambda$ are in fact the same.

We introduce a notion of \textit{overlap number} $o(\cl S, \hat \mu_\psi)$ associated to a $\Phi$-invariant Gibbs state $\hat\mu_\psi$ on $\Sigma_I^+ \times \Lambda$ (and to its $\pi_2$-projection $\mu_\psi$ on $\Lambda$), and we use thermodynamic formalism to relate it to the dimension of $\mu_\psi$.
In \textbf{Theorem \ref{fold}} and \textbf{Corollary \ref{maxent}} we show that the overlap number $o(\cl S, \hat \mu_\psi)$  is related to the folding entropy of $\hat\mu_\psi$ with respect to the lift map $\Phi$. 
In particular, this applies to Bernoulli measures on $\Sigma_I^+$ and their lifts on $\Sigma_I^+ \times \Lambda$. When $\mu=\mu_0$ is the projection of the measure of maximal entropy $\hat\mu_0$ from $\Sigma_I^+\times \Lambda$, one obtains a topological overlap number  $o(\cl S)$ of  $\cl S$, which quantifies the average level of overlapping in $\cl S$, and indicates how far is $\cl S$ from satisfying the Open Set Condition. By using Theorem \ref{equal}, we compute in \textbf{Corollary \ref{usorcalc}} the overlap number  $o(\cl S)$ as a limit of integrals over $\Sigma_I^+$ w.r.t the uniform Bernoulli measure $\nu_{(\frac {1}{|I|}, \ldots, \frac{1}{|I|})}$. And in general for Bernoulli measures $\nu_{\bf p}$, Corollary \ref{usorcalc} gives a simpler formula for $o(\cl S, \hat \mu_{\bf p})$.

Next,  in \textbf{Theorem \ref{dim}} we use the overlap number of  $\hat \mu_\psi$ to obtain
estimates for the Hausdorff dimension of a set of full $\mu_\psi$-measure in $\Lambda$, which set is constructed explicitly. This gives upper bounds for $HD(\mu_\psi)$, by using zeros of pressure functions associated to $o(\cl S, \hat \mu_\psi)$, which are computable in certain cases of interest. 

In \textbf{Section 3} we apply the results to the case of Bernoulli convolutions $\nu_\lambda$ for $\lambda \in (\frac 12, 1)$, where $\nu_\lambda$ gives the distribution of the random series $\mathop{\sum}\limits_{n\ge 0} \pm \lambda^n$ with the $+, -$ signs taken independently and with equal probabilities. In this case, one has an iterated function system with overlaps $\cl S_\lambda$, whose limit set is an interval $I_\lambda$, and $\nu_\lambda$ appears as the projection of the measure of maximal entropy $\nu_{(\frac 12, \frac 12)}$ from $\Sigma_2^+$ to $I_\lambda$. Bernoulli convolutions have attracted a lot of attention (see \cite{PSS}), starting with Erd\"os \cite{E} who showed that $\nu_\lambda$ is singular for $\lambda^{-1}$ Pisot; then, continuing with the result of Solomyak \cite{So} about the absolute continuity of $\nu_\lambda$ for Lebesgue-a.e $\lambda \in (\frac 12, 1)$, and the result of Przytycki and Urba\'nski \cite{PU} that $HD(\nu_\lambda) < 1$ for $\lambda^{-1}$ Pisot, and then with more recent results,  for example, by Hochman \cite{H} about the dimension of $\nu_\lambda$ for $\lambda$ outside a set of dimension zero in $(\frac 12, 1)$.

In \textbf{Theorem \ref{dimest}} we find a relation between $HD(\nu_\lambda)$ and the overlap number $o(\cl S_\lambda)$, for all $\lambda \in (\frac 12, 1)$. We  show how to approximate $o(\cl S_\lambda)$ with integrals on $\Sigma_2^+$ with respect to the uniform Bernoulli measure $\nu_{(\frac 12, \frac 12)}$.  By using known results on $HD(\nu_\lambda)$, one obtains then upper estimates for $o(\cl S_\lambda)$; in particular, one can estimate $o(\cl S_\lambda)$ more precisely for specific values of $\lambda$, like $\lambda = 2^{-\frac 1m}, m \ge 2$ (i.e $ \frac 1\lambda$ non-Pisot),  or $\lambda = \frac{\sqrt 5 -1}{2}$ (i.e $\frac 1\lambda$ Pisot). 
 In Corollary \ref{2} we prove that $o(\cl S_\lambda)$ is strictly less than 2, for all $\lambda \in (\frac 12, 1)$. 
In the end, we obtain dimension estimates for biased Bernoulli convolutions $\nu_{\lambda, p}$, for $\lambda \in (\frac 12, 1)$ and $p \in (0, 1)$. \ 
The results about overlap numbers can be applied also to other conformal iterated function systems with overlaps.

\section{Overlap numbers of measures and dimension estimates.}

First, let us define an \textit{overlap lift function} which allows to associate the dynamics of a map to our IFS $\cl S$. With regard to this function, the contractions $\phi_i$ appear as restrictions to cylinders $[i], i \in I$.  

\begin{defn}\label{olf}
In the above setting, for the finite IFS $\cl S = \{\phi_i\}_{i \in I}$, define the \textit{overlap lift map} 
$$\Phi: \Sigma_I^+ \times \Lambda \to \Sigma^+_I \times \Lambda, \ \Phi(\omega, x) = (\sigma \omega, \phi_{\omega_1}(x)), \ (\omega, x) \in \Sigma_I^+ \times \Lambda$$

\end{defn}

Let us now consider a H\"older continuous function $\psi: \Sigma_I^+\times \Lambda \to \bb R$. Since the  lift map $\Phi$ is distance-expanding in the first coordinate and contracting in the second coordinate, it follows that it is expansive and we can apply the theory of equilibrium states (for eg \cite{KH}, \cite{Wa}). As $\psi$ is H\"older, there exists a unique equilibrium measure for $\psi$ with respect to $\Phi$ on $\Sigma_I^+ \times \Lambda$, denoted by $\hat \mu_\psi$.

In particular, if we take a H\"older continuous function $g: \Lambda \to \mathbb R$ and the associated function $\psi_g: \Sigma_I^+\times \Lambda \to \mathbb R, \ \psi_g(\omega, x) = g(x)$, then we have the equilibrium measure $\hat\mu_{\psi_g}$ on $\Sigma_I^+\times \Lambda$ (relative to $\Phi$) and its projection $(\pi_2)_*(\hat\mu_{\psi_g})$ on $\Lambda$, where $\pi_2$ is the projection on the second coordinate. In general this measure is different from the projection $\pi_*(\bar \mu_{g\circ \pi})$, where $\pi: \Sigma_I^+ \to \Lambda, \pi(\omega) = \phi_{\omega_1}\circ \ldots$, and where in general $\bar \mu_\chi$ denotes the equilibrium measure of a H\"older continuous $\chi$ on $\Sigma_I^+$ (relative to the shift $\sigma$).

\

For any $n \ge 1$ and any $(\omega, x) \in \Sigma_I^+ \times \Lambda$, we have $\Phi^n(\omega, x) = (\sigma^n\omega, \phi_{\omega_n}\circ \phi_{\omega_{n-1}} \circ \ldots \circ \phi_{\omega_1}(x))$. 
Notice that, if $\eta_1, \ldots, \eta_n$ are given and if $\phi_{\omega_n}\circ \ldots \circ \phi_{\omega_1}(x) = \phi_{\eta_n}\circ \ldots \circ \phi_{\eta_1}(y)$, then from the injectivity of the contractions $\phi_i, i \in I$, there exists \textit{exactly one} point $y$ with this property. By Definition \ref{ov}, this means that, given the $n$-chain $(\eta_n, \ldots, \eta_1)$ as above, the corresponding $n$-root $y$ is uniquely defined such that $(\eta_n, \ldots, \eta_1)$ is an $n$-chain from $y$ to $\phi_{\omega_n\ldots \omega_1}(x)$. \ \    

Given now a measure $\hat \mu_\psi$ as above, an arbitrary point $(\omega, x) \in \Sigma_I^+ \times \Lambda$, and $\tau >0$, define the set of $n$-chains from points in $\Lambda$ to $\phi_{\omega_n\ldots\omega_1}(x)$, which are $\tau$-generic relative to $\hat \mu_\psi$: 
\begin{equation}\label{gen}
\Delta_n\big((\omega, x), \tau, \hat \mu_\psi\big)= \{(\eta_1, \ldots, \eta_n) \in I^n, \ \exists y \in \Lambda, \ \phi_{\eta_n\ldots\eta_1}(y) = \phi_{\omega_n\ldots\omega_1}(x) \ \text{and} \ |\frac{S_n\psi(\eta, y)}{n} - \int_{\Sigma_I^+\times \Lambda} \psi d\hat \mu_\psi| < \tau\},
\end{equation}
where $\eta = (\eta_1, \ldots, \eta_n, \omega_{n+1}, \omega_{n+2}, \ldots) \in \Sigma_I^+$, and where $S_n\psi(\eta, y) = \psi(\eta, y) + \psi(\Phi(\eta, y)) + \ldots +\psi(\Phi^n(\eta, y))$. We denote the cardinality of the set $\Delta_n$ by $b_n$, so $$b_n((\omega, x), \tau,  \hat \mu_\psi) := \text{Card} \ \Delta_n\big((\omega, x), \tau,  \hat \mu_\psi\big), \  \forall (\omega, x) \in \Sigma_I^+ \times \Lambda$$
 Remark that, if $(i_1, \ldots, i_n) \in \Delta_n\big((\omega, x), \tau, \hat\mu_\psi\big)$ with corresponding $n$-root $y$ of $\phi_{\omega_n\ldots\omega_1}(x)$, then $\Delta_n\Big(\big((i_1, \ldots, i_n, \omega_{n+1}, \omega_{n+2}, \ldots), y\big), \tau, \hat\mu_\psi\Big) = \Delta_n\Big((\omega, x), \tau, \hat\mu_\psi\Big)$.

\begin{defn}\label{of}
Given a H\"older continuous potential $\psi$ on $\Sigma_I^+\times \Lambda$ and $\tau>0$, we call $b_n(\cdot, \tau, \hat\mu_\psi): \Sigma_I^+\times \Lambda \to \mathbb N$ the $n$-\textbf{overlap function} associated to the measure $\hat\mu_\psi$ and $\tau$.
\end{defn} 

The function $b_n(\cdot, \tau, \hat\mu_\psi)$ is measurable and bounded, but in general discontinuous on $\Sigma_I^+\times \Lambda$.
In the sequel we will use the folding entropy of a $\Phi$-invariant measure $\mu$ on $\Sigma_I^+ \times \Lambda$; for general folding entropy see \cite{Ru-fold} (and for entropy production, also \cite{Ru-survey}, \cite{MU-JSP}). 
The folding entropy of a $\Phi$-invariant probability $\mu$ with respect to $\Phi: \Sigma_I^+ \times \Lambda \to \Sigma_I^+ \times \Lambda$,  is defined as the conditional entropy $F_{\Phi}(\mu) := H_{\mu}(\epsilon|\Phi^{-1}\epsilon)$, where $\epsilon$ is the point partition of the Lebesgue space $\Sigma_I^+ \times \Lambda$. \
In \cite{Pa} Parry introduced a notion of Jacobian of an invariant measure for an endomorphism, and studied its properties; in particular, the Jacobian satisfies the Chain Rule. Given a map $f:X \to X$ on a Lebesgue space $X$ and an $f$-invariant probability measure $\mu$, such that $f$ is essentially countable-to-one, we denote the Jacobian of $\mu$ by $J_f(\mu)$. From above and \cite{Pa} it follows that, in general, the folding entropy of a measure $\mu$ is equal to the integral of the logarithm of the Jacobian of $\mu$. So in our case,  the folding entropy of $\hat \mu_\psi$ with respect to $\Phi$  is given by: $$F_\Phi(\hat \mu_\psi) = \int_{\Sigma_I^+\times\Lambda}\log J_{\Phi}(\hat \mu_\psi) \ d\hat\mu_\psi$$

\

We investigate now the structure of the $\Phi$-invariant probabilities on the product space $\Sigma_I^+\times \Lambda$. 
Let define also the lift homeomorphism $\tilde \Phi$ on $\Sigma_I^+ \times \Sigma_I^+$, namely: $$\tilde \Phi:\Sigma_I^+\times \Sigma_I^+\to \Sigma_I^+\times \Sigma_I^+, \ \tilde \Phi(\omega, \eta) = (\sigma\omega, \omega_1\eta)$$  If $\tilde\pi(\omega, \eta):=(\omega, \pi(\eta))$, for $(\omega, \eta) \in \Sigma_I^+\times \Sigma_I^+$, then we obtain the following diagram of maps on $\Sigma_I^+\times \Sigma_I^+$, respectively $\Sigma_I^+\times \Lambda$, where both vertical maps below are equal to $\tilde \pi: \Sigma_I^+\times \Sigma_I^+ \to \Sigma_I^+\times \Lambda$:
\begin{equation}\label{dia}
\begin{array}{clclcr}
\Sigma_I^+\times \Sigma_I^+ & \ \ &\mathop{\longrightarrow}\limits^{\tilde \Phi} & \ \  &\Sigma_I^+\times \Sigma_I^+ \\
\downarrow &   & \ \ \ &  & \downarrow \\
\Sigma_I^+\times \Lambda & \ \ & \mathop{\longrightarrow}\limits^{\Phi} & \ \ & \Sigma_I^+\times \Lambda
\end{array}
\end{equation}
This diagram is commutative. Indeed, $\tilde\pi\circ\tilde\Phi(\omega, \eta) = (\sigma\omega, \pi(\omega_1\eta)=(\sigma\omega, \phi_{\omega_1}\circ\phi_{\eta_1}\circ\phi_{\eta_2}\circ\ldots)$; on the other hand, $\Phi\circ\tilde\pi(\omega, \eta) = \Phi(\omega, \phi_{\eta_1}\circ\phi_{\eta_2}\circ\ldots) = (\sigma\omega, \phi_{\omega_1}\circ\phi_{\eta_1}\circ\ldots)$. Hence $\tilde \pi \circ \tilde\Phi = \Phi\circ \tilde\pi$.

Also $\tilde \Phi$ is a homeomorphism. 
Then as in \cite{Ru-78}, by using Hahn-Banach Theorem and Markov-Kakutani Theorem and by approximating integrals of   functions  from $\mathcal{C}(\Sigma_I^+\times \Sigma_I^+, \mathbb R)$ with integrals of functions $g\circ\tilde\pi\circ\tilde \Phi^n, n\in \mathbb Z$, for $g \in \mathcal{C}(\Sigma_I^+\times \Lambda, \mathbb R)$, it follows that for any $\Phi$-invariant probability $\nu$ on $\Sigma_I^+\times \Lambda$, there exists a unique $\tilde \Phi$-invariant probability $\tilde \nu$ on $\Sigma_I^+\times \Sigma_I^+$ such that $\tilde \pi_*(\tilde\nu) = \nu$. 
In particular, the equilibrium measure $\hat\mu_\psi$ of the H\"older continuous $\psi$ on $\Sigma_I^+\times \Lambda$, is the $\tilde \pi$-projection of the equilibrium measure $\tilde\mu_{\tilde\psi}$ of $\tilde\psi:=\psi\circ\tilde \pi$ on $\Sigma_I^+\times \Sigma_I^+$.
Hence, the measure of maximal entropy $\hat\mu_0$ on $\Sigma_I^+\times \Lambda$ is the $\tilde\pi$-projection of the measure of maximal entropy $\tilde \mu_0$ for $\tilde \Phi$ on $\Sigma_I^+\times \Sigma_I^+$, i.e $$\hat\mu_0 = \tilde\pi_*(\tilde\mu_0)$$
Moreover,  the topological entropy of the map $\Phi$ is equal to the topological entropy of the shift $\sigma: \Sigma_I^+ \to \Sigma_I^+$, i.e $\log |I|$, because in the second coordinate we have contractions, so the separated sets are determined only by the expansion $\sigma$ in the first coordinate. With the canonical distance on $\Sigma_I^+$, $d(\omega, \eta) = \mathop{\sum}\limits_{i\ge 1}\frac{|\omega_i- \eta_i|}{2^i}$, the ball of center $\omega$ and radius $\frac{1}{2^n}$ is the cylinder $[\omega_1, \ldots, \omega_n]$, so $B((\omega, x), \frac{1}{2^n}) = [\omega_1, \ldots, \omega_n] \times B(x, \frac{1}{2^n})$.  
  If we consider $n$-roots of $x$ and the measure of maximal entropy $\hat\mu_0$ w.r.t $\Phi$, then all these $n$-roots are generic. Since in this case the overlap function $b_n$ does not depend on $\tau$, we  denote it simply by $b_n(\omega, x)$, for $(\omega, x) \in \Sigma_I^+\times \Lambda$.

\

In general, there are several ways to define \textbf{projections of invariant measures} on the fractal limit set $\Lambda$, depending whether we project $\sigma$-invariant measures on $\Sigma_I^+$, or $\Phi$-invariant measures on $\Sigma_I^+ \times \Lambda$, or $\tilde \Phi$-invariant measures on $\Sigma_I^+ \times \Sigma_I^+$. In many cases, for example for Bernoulli measures, these projections will be shown to coincide. 
Let us first consider a H\"older continuous potential $\psi$ on $\Sigma_I^+\times \Lambda$, and as above let  $\hat\mu_\psi$ its (unique) equilibrium state on $\Sigma_I^+\times \Lambda$;  if $\pi_2: \Sigma_I^+\times \Lambda \to \Lambda$ is the projection on the second coordinate $\pi_2(\omega, x) = x$,  denote  the projection measure on $\Lambda$ by:
\begin{equation}\label{mupsi}
\mu_\psi:=(\pi_2)_*(\hat\mu_\psi)
\end{equation}
Consider next $g$  a H\"older continuous potential on $\Sigma_I^+$, and let $\bar \mu_g$ be its unique equilibrium measure on $\Sigma_I^+$. 
Then we can define two kinds of projection measures on $\Lambda$. The first type is $\mu_\psi$ defined above in (\ref{mupsi}), where $\psi = g \circ \pi_1$; so $\mu_\psi = (\pi_2)_*(\hat \mu_\psi)$. The second type is the self-conformal measure:
\begin{equation}\label{sc}
\pi_*(\bar \mu_g),
\end{equation}
 where $\pi: \Sigma_I^+ \to \Lambda, \  \pi(\omega_1\omega_2\ldots) = \phi_{\omega_1}\circ \phi_{\omega_2} \circ\ldots$ \  is the canonical coding map for $\Lambda$. 

We now prove that, for \textbf{Bernoulli measures} on $\Sigma_I^+$, the two types of projection measures defined above, are in fact equal. This will make our results about overlap numbers apply to $\pi$-projections of Bernoulli measures onto $\Lambda$. Consider then a Bernoulli measure $\nu_{\mathbf{p}}$ on $\Sigma_I^+$ determined by an arbitrary probabilistic vector $\mathbf p = (p_1, \ldots, p_{|I|})$. Thus the $\nu_{\mathbf p}$-measure of the cylinder $[\omega_1, \ldots, \omega_n]= \{\eta \in \Sigma_I^+, \eta_1=\omega_1, \ldots, \eta_n = \omega_n\}$, is equal to $p_{\omega_1} \ldots p_{\omega_n} $ for any $n \ge 1$ and $\omega_i \in I, 1 \le i \le n.$ \
Consider the potential $\phi: \Sigma_I^+ \to \mathbb R, \ \phi(\omega_1\omega_2 \ldots) = \log p_{\omega_1}$, for $\omega = (\omega_1, \omega_2, \ldots) \in \Sigma_I^+$. 
Then $S_n\phi(\omega) = \phi(\omega) + \phi(\sigma(\omega)) + \ldots + \phi(\sigma^{n-1}(\omega)) = \log p_{\omega_1} \ldots p_{\omega_n}$. By taking Bowen balls for the shift $\sigma$ (which are cylinders in our case), we see immediately that $$P_\sigma(\phi) = 0$$  Clearly,  $\phi$ is H\"older continuous on $\Sigma_I^+$ and its unique equilibrium measure $\bar \mu_\phi$ is equal to the Bernoulli measure $\nu_{\mathbf p}$; this is  due to the expression of $\bar \mu_\phi$ on cylinders $[\omega_1 \ldots \omega_n]$ (see \cite{Bo}, \cite{KH}), i.e $$\frac 1C e^{S_n\phi(\omega) - nP_\sigma(\phi)} \le \bar\mu_\phi(B_n(\omega, \vp)) \le C e^{S_n\phi(\omega) - nP_\sigma(\phi)},$$ so we conclude that $$\bar \mu_\phi = \nu_{\mathbf p}$$
In case of Bernoulli measures, we can now prove that the various projection measures are \textbf{equal} on $\Lambda$:

\begin{thm}\label{equal}
In the above setting, let $\mathbf p = (p_1, \ldots, p_{|I|})$ an arbitrary probabilistic vector, and $\psi: \Sigma_I^+\times \Lambda \to \mathbb R, \ \psi((\omega_1\ldots), x) := \log p_{\omega_1}$, with $\hat\mu_{\psi}$ denoting the unique equilibrium measure of $\psi$ with respect to  $\Phi: \Sigma_I^+ \times \Lambda \to \Sigma_I^+ \times \Lambda$. Then the following measures  are equal on $\Lambda$:  
$$
 \pi_* \nu_{\mathbf p} = \pi_{2*} \hat\mu_{\psi} = (\pi_2 \circ \tilde \pi)_*(\nu_{\textbf{p}} \times \nu_{\textbf{p}}),
$$
where $\pi_2: \Sigma_I^+ \times \Lambda \to \Lambda, \ \pi_2(\omega, x) = x$,  and $\pi: \Sigma_I^+ \to \Lambda$ is the canonical coding map, and where $\tilde \pi: \Sigma_I^+ \times \Sigma_I^+ \to \Sigma_I^+ \times \Lambda, \ \tilde \pi(\omega, \eta) = (\omega, \pi(\eta))$.
\end{thm}

\begin{proof}
In order to prove the first equality, 
let us define $\tilde \psi = \psi \circ \tilde \pi$, where $\tilde \pi(\omega, \eta) = (\omega, \pi\eta)$. So $\tilde \psi$ is a H\"older continuous potential on $\Sigma_I^+ \times \Sigma_I^+$. Then recalling that $\tilde \Phi(\omega, \eta) = (\sigma\omega, \omega_1 \eta)$ is an expansive homeomorphism with specification property, it follows (\cite{KH}) that there exists a unique equilibrium measure $\tilde \mu_{\tilde \psi}$ on $\Sigma_I^+ \times \Sigma_I^+$. 
Also we have the projection $\tilde \pi(\omega, \eta) = (\omega, \pi \eta)$ from $\Sigma_I^+ \times \Sigma_I^+$ to $\Sigma_I^+ \times \Lambda$. Moreover, from definitions it can be seen that $$\tilde \pi \tilde \Phi(\omega, \eta) = (\sigma \omega, \phi_{\omega_1}(\pi \eta)) = \Phi\circ \tilde \pi(\omega, \eta),$$ so $\tilde \pi \circ \tilde \Phi = \Phi \circ \tilde \pi$. This implies that $\tilde \pi_*(\tilde \mu_{\tilde \psi}) = \hat \mu_\psi$, i.e the projection to  $\Sigma_I^+ \times \Lambda$ of the equilibrium measure of $\tilde \psi$ on $\Sigma_I^+ \times \Sigma_I^+$, is equal to the equilibrium measure of $\psi$.  Hence from above, 
\begin{equation}\label{mu2}
\pi_{2*}(\hat \mu_\psi)(A) = \hat \mu_\psi(\pi_2^{-1}(A)) = \tilde \mu_{\tilde \psi}(\Sigma_I^+\times \pi^{-1}(A))
\end{equation}
On the other hand, notice that the Bowen ball for $\tilde \Phi$ is given by $B_n((\omega, \eta), \vp) = [\omega_1 \ldots \omega_n] \times \Sigma_I^+$, and for any $1 \le i \le n$, we have $\tilde \Phi^i(B_n((\omega, \eta), \vp)) = [\omega_{i+1} \ldots \omega_n] \times [\omega_i \ldots \omega_1]$. From the $\tilde \Phi$-invariance of the equilibrium measure $\tilde \mu_{\tilde \psi}$, it follows that for any $1 \le i \le n$, 
\begin{equation}\label{i}
\tilde \mu_{\tilde \psi}(\tilde \Phi^i(B_n((\omega, \eta), \vp))) = \tilde \mu_{\tilde \psi}([\omega_1 \ldots \omega_n] \times \Sigma_I^+) = \tilde \mu_{\tilde \psi}([\omega_{i+1} \ldots \omega_n] \times [\omega_i \ldots \omega_1])
\end{equation}
However recall that $\pi_{1*}\hat \mu_\psi = \bar \mu_\phi = \nu_{\mathbf p}$, and thus $(\pi_1 \circ \tilde \pi)_* \tilde \mu_{\tilde \psi} = \nu_{\mathbf p}$. Therefore using also (\ref{i}) we obtain that,  for any $j \ge 1$ and any $\omega, \eta \in \Sigma_I^+$, 
\begin{equation}\label{produs1}
\tilde\mu_{\tilde \psi}([\omega_1] \times [\eta_1 \ldots \eta_j]) = \nu_{\mathbf p}([\eta_j \ldots \eta_1\omega_1]) = p_{\eta_j} \cdot \ldots \cdot p_{\eta_1} p_{\omega_1}
\end{equation}
By adding over $\omega_1 \in \Sigma_I^+$ we obtain that, for any $j \ge 1$ and for any $\eta = (\eta_1 \eta_2 \ldots) \in \Sigma_I^+$, 
$$\tilde \mu_{\tilde \psi}(\Sigma_I^+ \times [\eta_1 \ldots \eta_j]) = p_{\eta_1} \ldots p_{\eta_j} = \nu_{\mathbf p}([\eta_1 \ldots \eta_j]$$
But this works for any cylinder in $\Sigma_I^+$. Also, for any Borel set $A \subset \Lambda$, we have $\pi_{*}\nu_{\mathbf p}(A) = \nu_{\mathbf p}(\pi^{-1}(A))$.  Hence from the above, and by using also  (\ref{mu2}),  we can infer  that $\pi_{2*}\hat \mu_\psi$ is in fact a self-conformal measure on $\Lambda$, namely,
$$
\pi_{2*}\hat\mu_\psi = \pi_* \nu_{\mathbf p}
$$

We now prove the second equality. From before, $\tilde \Phi: \Sigma_I^+ \times \Sigma_I^+ \to \Sigma_I^+ \times \Sigma_I^+$ is a homeomorphism which preserves $\tilde \mu_{\tilde \psi}$. Also notice that for any $\omega_1, \omega_2, \eta_1, \ldots, \eta_m \in I$, one has $\tilde \Phi([\omega_1\omega_2] \times [\eta_1\ldots \eta_m]) = [\omega_2] \times [\omega_1\eta_1\eta_2\ldots \eta_m]$. 
But, from (\ref{produs1}), $\tilde \mu_{\tilde \psi}([\omega_2]\times[\omega_1\eta_1\ldots\eta_m]) = p_{\omega_2}p_{\omega_1}p_{\eta_1}\ldots p_{\eta_m}$, and from the $\tilde\Phi$-invariance of $\tilde\mu_{\tilde\psi}$,  it follows that $ \tilde\mu_{\tilde\psi}([\omega_1\omega_2] \times [\eta_1\ldots \eta_m]) =  \tilde\mu_{\tilde\psi}(\tilde\Phi([\omega_1\omega_2] \times [\eta_1\ldots \eta_m])) = p_{\omega_1}p_{\omega_2}p_{\eta_1}\ldots p_{\eta_m}$. Hence by induction it follows similarly that, for any $k, m \ge 1$,  $$\tilde\mu_{\tilde\psi}([\omega_1\ldots \omega_k]\times[\eta_1\ldots \eta_m]) = p_{\omega_1}\ldots p_{\omega_k} \cdot p_{\eta_1}\ldots p_{\eta_m}$$
This means that $\tilde\mu_{\tilde\psi} = \nu_{\textbf p} \times \nu_{\textbf p}$, \ and that $\pi_* \nu_{\textbf p} = (\pi_2 \circ \tilde\pi)_*(\nu_{\textbf p} \times \nu_{\textbf p})$. 

\end{proof}

The equality of the projection measures for Bernoulli probabilities has useful consequences when computing the associated overlap numbers, see Corollary \ref{usorcalc}.

\

For any conformal  iterated function system $\cl S$, we want to prove now that the exponential rate of growth in $n$, of the number of generic $n$-chains/roots from $\Delta_n$, is approaching the folding entropy of the measure $\hat \mu_\psi$. In particular it follows that, on average, the number of $n$-chains associated to the $n$-overlaps of $\Lambda$ grows exponentially like $e^{nF_\Phi(\hat\mu_0)}$.  

\begin{thm}\label{fold}
Let a finite conformal IFS $\cl S = \{\phi_i, i \in I\}$ with limit set $\Lambda$, and a H\"older continuous potential $\psi$ on the lift space $\Sigma_I^+ \times \Lambda$; denote the equilibrium measure of $\psi$ on $\Sigma_I^+\times \Lambda$ by $\hat \mu_\psi$. Then, $$\mathop{\lim}\limits_{\tau \to 0}\mathop{\lim}\limits_{n \to \infty} \frac 1n \int_{\Sigma_I^+\times\Lambda} \log b_n((\omega, x), \tau, \hat \mu_\psi) \ d\hat\mu_\psi(\omega, x) = F_\Phi(\hat\mu_\psi)$$
\end{thm} 

\begin{proof}
In our case the map $\Phi: \Sigma_I^+ \times \Lambda \to \Sigma_I^+\times \Lambda$ is distance-expanding in the first coordinate, and distance contracting in the second coordinate. 
Let $B_m(z, \vp)$ denote the $(m, \vp)$-Bowen ball around $z$ in the canonical product metric on the compact metric space $\Sigma_I^+\times \Lambda$ with respect to the endomorphism $\Phi$; hence in particular it is expansive. 
Since $\hat \mu_\psi$ is the equilibrium measure of a H\"older continuous potential on $\Sigma_I^+\times \Lambda$, we can apply the properties of equilibrium measures with respect to expansive maps on compact metric spaces (see \cite{KH}). \ We will use first the ideas of Theorem 1 from \cite{M-ETDS11}, giving the comparison between the (equilibrium) measure of various parts of the preimage set. So, from \cite{M-ETDS11} there exists a constant $C>0$ such that, for any positive integer $m$ and for any sets $A_1, A_2$ satisfying $A_1 \subset B_m(z_1, \vp), A_2 \subset B_m(z_2, \vp)$ and $\Phi^m(A_1) = \Phi^m(A_2)$, we have:
\begin{equation}\label{comp}
\frac{1}{C}\frac{\hat\mu_\psi(A_2)}{e^{S_m\psi(z_2)}} \le \frac{\hat\mu_\psi(A_1)}{e^{S_m\psi(z_1)}} \le C \frac{\hat\mu_\psi(A_2)}{e^{S_m\psi(z_2)}}
\end{equation}

Now the Jacobian of the measure $\hat \mu_\psi$ with respect to $\Phi^n$ gives the change in the measure of a set by applying the map $\Phi^n$ (see \cite{Pa}); hence for any integer $n \ge 1$,  $\hat\mu_\psi(\Phi^n(\cl A)) = \int_{\cl A}J_{\Phi^n}(\hat \mu_\psi) d\hat \mu_\psi$, for any measurable set $\cl A \subset \Sigma_I^+\times \Lambda$, on which $\Phi^n$ is injective. But in fact, $J_{\Phi^n}(\hat\mu_\psi)(\omega, x) = \mathop{\lim}\limits_{r \to 0} \frac{\hat\mu_\psi(\Phi^n(B((\omega, x), r)}{\hat\mu_\psi(B((\omega, x), r)}$, for $\hat\mu_\psi$-a.e $(\omega, x) \in \Sigma_I^+\times \Lambda$. However from the $\Phi$-invariance of the measure $\hat \mu_\psi$ it follows that $\hat\mu_\psi(\Phi^{n}(\cl A)) = \hat\mu_\psi(\Phi^{-n}(\Phi^n(\cl A)))$, for any Borel set $\cl A$. Hence we can apply the above comparison between the various parts of the preimage set $\Phi^{-n}(\Phi^n(\cl A))$ for $n$ arbitrary (i.e in fact the comparison between various sets taken by different compositions  $\phi_{j_1}\circ\ldots\circ\phi_{j_n}$ to the same image),  in order to obtain that there exists a constant $C>0$ independent of $n$ such that:
 
\begin{equation}\label{jaco}
\ \frac{\mathop{\sum}\limits_{(\eta, y), \Phi^n(\eta, y) = \Phi^n(\omega, x)} \exp(S_n\psi(\eta, y))}{C\cdot \exp(S_n\psi(\omega, x))} \le J_{\Phi^n}(\hat \mu_\psi)(\omega, x) \le C \cdot \frac{\mathop{\sum}\limits_{(\eta, y), \Phi^n(\eta, y) = \Phi^m(\omega, x)} \exp(S_n\psi(\eta, y))}{\exp(S_n\psi(\omega, x))},
\end{equation}
for $\hat\mu_\psi$-a.e pair $(\omega, x) \in \Sigma_I^+\times \Lambda$. 
Now, as the probability $\hat \mu_\psi$ is $\Phi$-invariant on the product space $\Sigma_I^+ \times \Lambda$, it follows from (\ref{jaco}) and from the properties of the folding entropy that 
\begin{equation}\label{fe}
\begin{aligned}
F_\Phi(\hat \mu_\psi) = &\frac{1}{n} \int_{\Sigma_I^+\times \Lambda}\log J_{\Phi^n}(\hat\mu_\psi)(\omega, x) d\hat \mu_\psi(\omega, x) = \\
&=\mathop{\lim}\limits_{n \to \infty} \frac 1n\int_{\Sigma_I^+\times \Lambda} \log \frac{\mathop{\sum}\limits_{\Phi^n(\eta, y) = \Phi^n(\omega, x)} \exp(S_n\psi(\eta, y))}{\exp(S_n\psi(\omega, x))} d\hat \mu_\psi(\omega, x)
\end{aligned}
\end{equation}

From Birkhoff Ergodic Theorem we know that, \  $\hat \mu_\psi((\omega, x) \in \Sigma_I^+\times \Lambda, |\frac{S_n\psi(\omega, x)}{n}-\int_{\Sigma_I^+\times \Lambda} \psi d\hat \mu_\psi| > \tau/2) \mathop{\to}\limits_{n\to \infty} 0$. Then, for any positive small number $\xi$, there exists an integer $n = n(\xi) \ge 1$ so that for all integers $n \ge n(\xi)$, we have
\begin{equation}\label{birk}
\mu_\psi((\omega, x) \in \Sigma_I^+\times \Lambda, |\frac{S_n\psi(\omega, x)}{n}-\int_{\Sigma_I^+\times \Lambda} \psi d\hat \mu_\psi| > \tau/2) < \xi
\end{equation}
Recall that, if $(\eta_1, \ldots, \eta_n) \in \Delta_n((\omega, x), \tau, \hat\mu_\psi)$, then the $n$-chain $(\eta_n, \ldots, \eta_1)$ uniquely determines an $n$-root $y$ of $\phi_{\omega_n\ldots\omega_1}(x)$. Hence with $\eta_{n+i} = \omega_{n+i}, i \ge 1$,  we can consider also the finite set $$\Delta_n'((\omega, x), \tau, \hat\mu_\psi) = \{(\eta, y) \in \Sigma_I^+ \times \Lambda, \  \Phi^n(\eta, y) = \Phi^n(\omega, x), \ |\frac{S_n\psi(\eta, y)}{n} - \int\psi \ d\hat\mu_\psi| < \tau\},$$
and there exists a bijection between $\Delta_n((\omega, x), \tau, \hat\mu_\psi)$ and $\Delta_n'((\omega, x), \tau, \hat\mu_\psi)$, taking $(\eta_1, \ldots, \eta_n)$ to $((\eta_1, \ldots, \eta_n, \omega_{n+1}, \omega_{n+2}, \ldots), y)$.  Thus $b_n((\omega, x), \tau, \hat\mu_\psi) = \text{Card}\Delta_n'((\omega, x), \tau, \hat\mu_\psi)$.
We now define the following set of $n$-roots, $$\Gamma_n((\omega, x), \tau, \hat \mu_\psi):=\{(\eta, y) \in \Sigma_I^+\times \Lambda, \Phi^n(\eta, y) = \Phi^n(\omega, x), (\eta_1, \ldots, \eta_n) \notin \Delta_n((\omega, x), \tau, \hat \mu_\psi)\}$$  Denote the sum corresponding to the roots from $\Gamma_n((\omega, x), \tau, \hat\mu_\psi)$ by $$\vartheta_n((\omega, x), \tau, \hat\mu_\psi):= \mathop{\sum}\limits_{(\eta, y) \in \Gamma_n((\omega, x), \tau, \hat\mu_\psi)} \exp(S_n\psi(\eta, y))$$
Let us now see what a typical Bowen ball for the map $\Phi: \Sigma_I^+\times \Lambda \to \Sigma_I^+\times \Lambda$ looks like. If $d(\cdot, \cdot)$ denotes the product metric, and if $d(\Phi^i(\omega, x), \Phi^i(\eta, y)) < \vp, 0 \le i \le n-1$, then there exists an integer $N(\vp)$ so that $\omega_i=\eta_i, i= 1, \ldots, n+N(\vp)$, and $d(x, y)<\vp$, since the maps $\phi_j$ are all contractions.
For an arbitrary $n \ge 2$, we now consider a measurable partition of $\Sigma_I^+\times\Lambda$ modulo $\hat\mu_\psi$, into sets $L_i^n, 1 \le i \le p_n$, such that for any $1 \le i \le p_n$ there exists a point $\zeta_i \in L_i^n$ so that for any point $\zeta_{ij} \in \Phi^{-n}(\zeta_i), 1 \le j \le p_{i, n}$, we have $L_i^n \subset \Phi^n(B_n(\zeta_{ij}, \vp))$. The integer $p_{i, n} \ge 1$ depends on $i$ for $1 \le i \le p_n$, and it is given by the number of $n$-roots of $\zeta_i$ in $\Lambda$, with respect to $\cl S$. This is possible to do if we take the sets $L_i^n$ small enough. Then, let us denote by $L_{ij}^n:= \Phi^{-n}(L_i^n)\cap B_n(\zeta_{ij}, \vp)$, for $1 \le i \le p_n, 1\le j \le p_{i, n}$.  \ 
Notice that  if $\Phi(\eta, y) = \Phi(\eta', y') = (\omega, x) \in \Sigma_I^+\times \Lambda$, then $\sigma\eta=\sigma\eta'=\omega$, i.e $\eta_2 =\omega_2, \ldots$, and $\phi_{\eta_1}(y) = \phi_{\eta_1'}(y')=x$. If $\eta_1 \neq \eta_1'$, then $d((\eta, y), (\eta', y')) \ge d(\eta_1, \eta_1') >\vp_0>\vp$, for some $\vp_0>0$. If $\eta_1 = \eta_1'$, then $\phi_{\eta_1}(y) = \phi_{\eta_1'}(y')$; but $\phi_\eta, \eta \in I$ are injective and thus $y = y'$. This implies that the sets $L_{ij}^n$ are mutually disjoint in $i,j$. We now decompose the integral of the logarithm of the Jacobian of $\hat \mu_\psi$ with respect to $\Phi^n$, along this partition with the sets $L_{ij}^n, 1 \le i \le p_n, 1 \le j \le p_{i, n}$. 
Therefore, for an arbitrary $n \ge 2$, we have:
\begin{equation}\label{decomp}
\int_{\Sigma_I^+\times\Lambda}\log\frac{\mathop{\sum}\limits_{\Phi^n(\eta, y) = \Phi^n(\omega, x)} \exp(S_n\psi(\eta, y))}{\exp(S_n\psi(\omega, x))} d\hat\mu_\psi (\omega, x)= \mathop{\mathop{\sum}\limits_{1\le i\le p_n}}\limits_{ 1\le j\le p_{i, n}} \int_{L_{ij}^n} \log\frac{\mathop{\sum}\limits_{\Phi^n(\eta, y) = \Phi^n(\omega, x)} \exp(S_n\psi(\eta, y))}{\exp(S_n\psi(\omega, x))} d\hat\mu_\psi(\omega, x) 
\end{equation}
Now, in regards to formula (\ref{jaco}),  we can write in general $$\mathop{\sum}\limits_{(\eta, y) \in \Phi^{-n}\Phi^n(\omega, x)} e^{S_n\psi(\eta, y)} = \mathop{\sum}\limits_{(\eta_1, \ldots, \eta_n) \in \Delta_n((\omega, x), \tau, \hat\mu_\psi)} e^{S_n\psi(\eta, y)} + \vartheta_n((\omega, x), \tau, \hat \mu_\psi)$$
Denote also $\rho_n(i, \tau, \hat\mu_\psi):= \mathop{\sum}\limits_{j, \zeta_{ij} \notin \Delta_n'(\zeta_{i1}, \tau, \hat\mu_\psi)} \hat\mu_\psi(L_{ij}^n)$. 
Thus by using (\ref{comp}), the definition of $\Delta_n'((\omega, x), \tau, \hat \mu_\psi)$ and the fact that $b_n((\omega, x), \tau, \hat\mu_\psi) = \text{Card}(\Delta_n'((\omega, x), \tau, \hat\mu_\psi))$, we obtain that the above sum in (\ref{decomp}) is comparable to the sum: $$\mathop{\sum}\limits_{i, j} \hat\mu_\psi(L_{ij}^n) \log\frac{b_n(\zeta_{ij}, \tau, \hat\mu_\psi) \hat\mu_\psi(L_{ij}^n) +\rho_n(i, \tau, \hat\mu_\psi)}{\hat\mu_\psi(L_{ij}^n)},$$ where we recall that the comparability constant $C>0$ does not depend on $n$, nor on $L_{i j}^n$.  
Now in general, if $(\eta, y)\in \Delta_n'((\omega, x), \tau, \hat\mu_\psi)$, and if $0 < \vp < \tau$ and $(\eta, y) \in B_n(\zeta_{ij}, \vp)$, then since the potential $\psi$ is H\"older continuous, it follows that $$\Big|\frac{S_n\psi(\eta, y)}{n} - \frac{S_n\psi(\zeta_{ij})}{n}\Big| \le v(\tau),$$ for some small $v(\tau)>0$ where $\mathop{\lim}\limits_{\tau \to 0}v(\tau) = 0$. Also, if $K:= \sup_{\Sigma_I^+\times \Lambda}|\psi|$, then $e^{S_n\psi(\eta, y)} \le e^{n K}$. Notice in addition, that the set $\Phi^{-n}\Phi^n(\omega, x)$ has at most $|I|^n$ elements in $\Sigma_I^+ \times \Lambda$. Denote the set of indices $j$ corresponding to nongeneric roots by $Q(n, i, \tau, \hat\mu_\psi) := \{j, 1 \le j \le p_{i, n}, \ \zeta_{ij} \in \Gamma_n(\zeta_{i1}, \tau, \hat\mu_\psi)\}$.  Then if $j \in Q(n, i, \tau, \hat\mu_\psi)$, then $\frac 1n |S_n\psi(\zeta_{ij})- \int_{\Sigma_I^+\times \Lambda} \psi d\hat \mu_\psi| > \tau$. Hence we can use the measure estimate in (\ref{birk}) to obtain that:  $$\mathop{\sum}\limits_{1 \le i \le p_n, \ j \in Q(n, i, \tau, \hat\mu_\psi)} \frac 1n \int_{L_{ij}^n} \log\frac{\mathop{\sum}\limits_{(\eta, y) \in \Phi^{-n}\Phi^n(\omega, x)} \exp(S_n\psi(\eta, y))}{\exp(S_n\psi(\omega, x))} d\hat\mu_\psi(\omega, x) \le \frac 1n \xi \log(2K|I|^n)$$ 
Therefore, from the comparison in (\ref{comp}) and from the above discussion, it follows that there exists a positive constant $C$, independent of $n$, of the partition $\{L_i^n\}_{1\le i\le p_n}$ and of the points $\zeta_i\in L_i^n$,  such that:
\begin{equation}\label{v}
\begin{aligned}
&\frac{1}{n}\mathop{\mathop{\sum}\limits_{1\le i\le p_n}}\limits_{ j \notin Q(n, i, \tau, \hat\mu_\psi)} \hat\mu_\psi(L_{ij}^n) \log b_n(\zeta_{i1}, \tau, \hat\mu_\psi) + \frac{1}{n}  \mathop{\sum}\limits_{i, j\notin Q(n, i, \tau, \hat\mu_\psi)} \hat\mu_\psi(L_{ij}^n)\log(1+\frac{\rho_n(i, \tau, \hat\mu_\psi)}{b_n(\zeta_{i1}, \tau, \hat\mu_\psi)\hat\mu_\psi(L_{ij}^n)}) - v(\tau)-C\xi  \\
&\le \int_{\Sigma_I^+\times \Lambda} \frac 1n \log\frac{\mathop{\sum}\limits_{(\eta, y) \in \Phi^{-n}\Phi^n(\omega, x)} \exp(S_n\psi(\eta, y))}{\exp(S_n\psi(\omega, x))} \ d\hat\mu_\psi(\omega, x) \le \\
&\le \frac{1}{n}\mathop{\mathop{\sum}\limits_{1\le i\le p_n}}\limits_{ j \notin Q(n, i, \tau, \hat\mu_\psi)} \hat\mu_\psi(L_{ij}^n) \log b_n(\zeta_{i1}, \tau, \hat\mu_\psi) + \frac{1}{n}  \mathop{\sum}\limits_{i, j\notin Q(n, i, \tau, \hat\mu_\psi)} \hat\mu_\psi(L_{ij}^n)\log(1+\frac{\rho_n(i, \tau, \hat\mu_\psi)}{b_n(\zeta_{i1}, \tau, \hat\mu_\psi)\hat\mu_\psi(L_{ij}^n)}) +v(\tau) +C\xi,
\end{aligned}
\end{equation}
where we recall that $\xi$ is the bound on the measure of non-generic points in (\ref{birk}).
But in general,  $\log(1+x) \le x$ for any $x >0$, hence $\log(1+\frac{\rho_n(i, \tau, \hat\mu_\psi)}{b_n(\zeta_{i1}, \tau, \hat\mu_\psi)\hat\mu_\psi(L_{ij}^n)}) \le 
\frac{\rho_n(i, \tau, \hat\mu_\psi)}{b_n(\zeta_{i1}, \tau, \hat\mu_\psi)\hat\mu_\psi(L_{ij}^n)}$. Therefore from (\ref{birk}), the second sum in the right-hand term of (\ref{v}) is less than $\xi$, which implies that: $$
\begin{aligned}
\Big|\frac 1n \int_{\Sigma_I^+\times \Lambda} & \frac 1n \log\frac{\mathop{\sum}\limits_{(\eta, y) \in \Phi^{-n}\Phi^n(\omega, x)}  \exp(S_n\psi(\eta, y))}{\exp(S_n\psi(\omega, x)} d\hat\mu_\psi(\omega, x) - \frac 1n \int_{\Sigma_I^+\times \Lambda} \log b_n((\omega, x), \tau, \hat\mu_\psi) d\hat\mu_\psi(\omega, x)\Big| \\
&\le v(\tau) +C\xi
\end{aligned}
$$
Therefore, using the expression for the folding entropy $F_\Phi(\hat\mu_\psi)$ from (\ref{fe}), and the fact that $\xi$ converges to 0 when $\tau$ converge to 0 (and also that $v(\tau)$ converges to 0 at the same time), we obtain the conclusion of the Theorem.

\end{proof}

We now want to define a notion of overlap number of $\cl S$ associated to an equilibrium state $\hat\mu_\psi$. This notion will take into consideration the $\hat\mu_\psi$-generic $n$-roots in $\Lambda$ and all the corresponding $n$-chains starting from them, for $n$ large. In particular, we obtain a (topological) overlap number of the system $\cl S$, which gives the average rate of growth of the number of $n$-chains from $n$-roots to points in $\Lambda$. 

\begin{cor}\label{maxent}
If $\cl S = \{\phi_i, i \in I\}$ is an arbitrary finite conformal iterated function system with overlaps and $\Lambda$ is its limit set, and if $\psi$ is a H\"older continuous potential on $\Sigma_I^+ \times \Lambda$ with equilibrium measure $\hat \mu_\psi$, we  call the \textbf{overlap number of $\cl S$ with respect to $\hat\mu_\psi$}, 
\begin{equation}\label{onmu}
o(\cl S, \hat \mu_\psi):= \exp\big(\mathop{\lim}\limits_{\tau \to 0}\mathop{\lim}\limits_{n\to \infty}
\frac 1n \int_{\Sigma_I^+\times \Lambda} \log b_n((\omega, x), \tau, \hat\mu_\psi) \ d\hat\mu_\psi(\omega, x)\big)
\end{equation}
If $\hat\mu_0$ is the measure of maximal entropy for $\Phi$ on $\Sigma_I^+\times\Lambda$, then the (topological) \textbf{overlap number of $\cl S$} is given by:
 $$
\begin{aligned}
o(\cl S) := o(\cl S, \hat\mu_0) &= \exp\big(\mathop{\lim}\limits_{n \to \infty} \frac 1n \int_{\Sigma_I^+\times \Lambda} \log b_n(\omega, x) \ d\hat\mu_0(\omega, x)\big) = \exp\big(F_\Phi(\hat\mu_0)\big) =\\ &=\exp\big(\int_{\Sigma_I^+\times\Lambda} \log \mathop{\lim}\limits_{n\to\infty} \frac{\hat\mu_0([\omega_2, \ldots, \omega_n] \times \phi_{\omega_1}(B(x, \frac{1}{2^n}))}{\hat\mu_0([\omega_1, \ldots, \omega_n])\times B(x, \frac{1}{2^n}))} \ d\hat\mu_0(\omega, x)\big)
\end{aligned}
$$
\end{cor}

\

\

In the case of \textbf{projections of Bernoulli measures}, we can use now Theorem \ref{equal} to compute more easily the overlap numbers.
Let us take an arbitrary probability vector $\textbf p = (p_1, \ldots, p_{|I|})$, which gives a Bernoulli measure $\nu_{\bf p}$ on $\Sigma_I^+$. 
According to the discussion before Theorem \ref{equal}, there exists an equilibrium measure denoted $\hat \mu_{\bf p}$ of the potential $\psi((\omega_1, \ldots), x)= \log p_{\omega_1}, \ (\omega, x) \in \Sigma_I^+\times \Lambda$, with respect to $\Phi$ on $\Sigma_I^+\times \Lambda$, so that $\pi_*\nu_{\bf p} = \pi_{2*}\hat\mu_{\bf p}$. The measure $\hat\mu_{\bf p}$ is called the equilibrium measure (with respect to $\Phi$) associated to $\bf p$. Denote also by $h(\textbf p):= \mathop{\sum}\limits_{1 \le j \le |I|} p_j \log p_j$, and notice that $h(\bf p) =
 \int\psi \ d\hat \mu_{\bf p}$.
Let us denote now by $$\beta_n(x):= \text{Card}\{(\eta_1, \ldots, \eta_n) \in I^n, \ x \in \phi_{\eta_1}\circ \ldots \circ \phi_{\eta_n}(\Lambda)\}, \ \forall x \in \Lambda$$
More generally, we define for $\tau>0$, 
\begin{equation}\label{betap}
\beta_n(x, \tau,  \textbf p):= \text{Card}\{(\eta_1, \ldots, \eta_n) \in I^n, \ x \in \phi_{\eta_1}\circ \ldots \circ \phi_{\eta_n}(\Lambda), \ |\frac{\log(p_{\eta_1}\ldots p_{\eta_n})}{n} - h(\bf p)| < \tau \}
\end{equation}
As before if $x \in \phi_{\eta_1}\circ \ldots \circ \phi_{\eta_n}(\Lambda)$, then there exists a unique point $y \in \Lambda$ with $x = \phi_{\eta_1}\circ \ldots \circ \phi_{\eta_n}(y)$. When the system $\cl S$ satisfies Open Set Condition, then the overlap number $o(\cl S, \hat\mu_{\bf p})$ is equal to 1.

We prove now the following simpler expression for the overlap number in the case of Bernoulli projections for conformal IFS's with overlaps $\cl S$, by employing the function $\beta_n(\cdot)$, that counts the number of $n$-chains from $n$-roots in the limit set $\Lambda$:

\begin{cor}\label{usorcalc}
Let a conformal iterated function system with overlaps $\cl S = \{\phi_i, i \in I\}$ with limit set $\Lambda$, and consider $\textbf p$ an arbitrary  probabilistic vector, with $\hat\mu_{\bf p}$ being the equilibrium measure on $\Sigma_I^+ \times \Lambda$ associated to $\bf p$. Then, the overlap number $o(\cl S, \hat \mu_{\bf p})$ can be computed as:
$$o(\cl S, \hat\mu_{\bf p}) = \exp\Big(\mathop{\lim}\limits_{\tau \to 0}\mathop{\lim}\limits_n \frac 1n \int_{\Sigma_I^+} \log \beta_n(\pi\omega, \tau, \textbf p) \ d\nu_{\bf p}(\omega)\Big)$$

In particular, we obtain the (topological) overlap number of $\cl S$, by integrating with respect to the uniform Bernoulli measure $\nu_{(\frac{1}{|I|}, \ldots, \frac{1}{|I|})}$,
$$o(\cl S) = \exp\Big(\mathop{\lim}\limits_n \frac 1n \int_{\Sigma_I^+} \log \beta_n(\pi\omega) \ d\nu_{(\frac{1}{|I|}, \ldots, \frac{1}{|I|})}(\omega)\Big)$$
\end{cor}

\begin{proof}
We prove here  the second part of the statement, about the topological overlap number; the first part follows similarly.
Let us denote by $\textbf p = (\frac {1}{|I|}, \ldots, \frac{1}{|I|})$, and consider $\mu_{\bf p} = \pi_*\nu_{\bf p}$. As in Theorem \ref{equal} there exists a corresponding $\Phi$-invariant measure $\hat \mu_{\bf p}$ on $\Sigma_I^+ \times \Lambda$.  
We have from Theorem \ref{equal} that $\pi_*\nu_{\bf p} = \pi_{2*}\hat\mu_{\bf p}$, hence 
$$\int_{\Lambda} \log\beta_n(x) \ d\mu_{\bf p}(x) = \int_{\Sigma_I^+\times \Lambda}\log\beta_n \circ \pi_2(\omega, x) \ d\hat\mu_{\bf p}(\omega, x) = \int_{\Sigma_I^+\times \Lambda}\log\beta_n \circ \pi_2\circ \Phi^n(\omega, x) \ d\hat\mu_{\bf p}(\omega, x)$$
But notice that $\beta_n\circ \pi_2\circ\Phi^n(\omega, x) = \beta_n(\phi_{\omega_n}\circ \ldots \circ \phi_{\omega_1}(x)) = \text{Card}\{(\eta_1, \ldots, \eta_n) \in I^n, \ \phi_{\omega_n}\circ \ldots \circ \phi_{\omega_1}(x) \in \phi_{\eta_1} \circ \ldots \circ \phi_{\eta_n}(\Lambda) \} = b_n(\omega, x)$, for any $(\omega, x)$.
Therefore, from the last displayed equality, it follows that:
$$\int_{\Sigma_I^+} \log \beta_n(\pi\omega) \ d\nu_{(\frac{1}{|I|}, \ldots, \frac{1}{|I|})} (\omega) = \int_\Lambda\log \beta_n(x)\  d\mu_{\bf p}(x) = \int_{\Sigma_I^+\times \Lambda} \log b_n(\omega, x) \ d\hat\mu_{\bf p}(\omega, x)$$

\end{proof}

We now show that overlap numbers of conformal IFS and of equilibrium measures on $\Sigma_I^+ \times \Lambda$, can be used to estimate the dimensions of the associated projection measures on $\Lambda$.  
Denote the Hausdorff dimension (for sets or measures) by $HD$. Recall that, in general for a measure $\mu$ on a metric space $X$,  its Hausdorff dimension is defined by: $$HD(\mu):= \inf\{HD(Z), Z \subset X \ \text{with} \ \mu(X\setminus Z)=0\}$$ In the following Theorem, we give an upper estimate for $HD(\mu_\psi)$, by estimating $HD(\Lambda \setminus Z(\psi))$ for some set $Z(\psi)\subset \Lambda$ of $\mu_\psi$-measure zero  with the help of the overlap number $o(\cl S, \hat \mu_\psi)$. Moreover, we will construct explicitly this set of $\mu_\psi$-measure zero $Z(\psi)$ below.

\begin{thm}\label{dim}
Consider a finite conformal iterated function system $\cl S = \{\phi_i\}_{i \in I}$ with limit set $\Lambda$, $\pi: \Sigma_I^+\to \Lambda$ be the canonical projection, and let a H\"older continuous potential $\psi: \Sigma_I^+\times \Lambda \to \mathbb R$, with its (unique) equilibrium measure $\hat \mu_\psi$; and let $\mu_\psi:= \pi_{2*}\hat\mu_\psi$ be the projection as  in (\ref{mupsi}). Then, $$HD(\mu_\psi) \le t(\cl S, \psi),$$
where $t(\cl S, \psi)$ is the unique zero of the pressure function with respect to the shift $\sigma: \Sigma_I^+ \to \Sigma_I^+$,  $$t \to P_\sigma(t\log|\phi_{\omega_1}'(\pi(\sigma\omega))| - \log o(\cl S, \hat\mu_\psi))$$

\end{thm}

\begin{proof}

Let denote by $R_n(\hat\mu_\psi, \delta)$ the set of points $(\omega, x) \in \Sigma_I^+\times \Lambda$ for which the number of generic roots satisfies $b_n((\omega, x), \tau, \hat\mu_\psi) < \frac 12\cdot e^{n(F_\Phi(\hat\mu_\psi)-\delta)}$. We want to show that the $\hat\mu_\psi$-measure of these sets converges to 0, when $n\to \infty$. If this does not happen, then there exist an infinite sequence $\{k_n\}_n$ and a number $\beta >0$, such that $\hat\mu_\psi(R_{k_n}(\hat \mu_\psi, \delta)) > \beta >0, \forall n \ge 1$. 
Then, for all pairs $(\omega, x) \in R_{k_n}(\hat\mu_\psi, \delta)$, $$\frac{\log b_{k_n}((\omega, x), \tau, \hat \mu_\psi)}{k_n} < \frac{-\log 2}{k_n} + F_\Phi(\hat\mu_\psi) - \delta$$
Therefore, after integrating with respect to $\hat\mu_\psi$, 
$$\int_{R_{k_n}(\hat\mu_\psi, \delta)} \frac{\log b_{k_n}((\omega, x), \tau, \hat\mu_\psi)}{k_n} d\hat \mu_\psi(\omega, x) < \hat\mu_\psi(R_{k_n}(\hat\mu_\psi, \delta))\cdot (F_\Phi(\hat\mu_\psi)-\delta -\frac{\log 2}{k_n})$$
We now use the last displayed inequality, and the properties of $J_{\Phi^n}(\hat\mu_\psi)$ from the proof of Theorem \ref{fold} (namely relation  (\ref{jaco})); thus by adding the integral of $\frac{\log b_{k_n}((\omega, x), \tau, \hat\mu_\psi)}{k_n}$ over $R_{k_n}$ and the integral of $\frac{\log b_{k_n}((\omega, x), \tau, \hat\mu_\psi)}{k_n}$ over the complement of $R_{k_n}$, we  obtain that: 
\begin{equation}\label{star}
\begin{aligned}
\int_{\Sigma_I^+\times \Lambda} \frac{\log b_{k_n}((\omega, x), \tau, \hat\mu_\psi)}{k_n} & d\hat\mu_\psi(\omega, x)  < \hat\mu_\psi(R_{k_n}(\hat\mu_\psi, \delta))\cdot (F_\Phi(\hat\mu_\psi)-\delta-\frac{\log 2}{k_n}) \ +\\ &+ \int_{\Sigma_I^+\times \Lambda \setminus R_{k_n}(\hat\mu_\psi, \delta)} \frac{\log J_{\Phi^{k_n}}(\hat\mu_\psi)}{k_n} d\hat\mu_\psi(\omega, x)
\end{aligned}
\end{equation}
On the other hand, from the Chain Rule we know that $\log J_{\Phi^n}(\hat\mu_\psi)(\omega, x) = \log J_\Phi(\omega, x) + \ldots + \log J_\Phi(\hat\mu_\psi)(\Phi^{n-1}(\omega, x))$, for all $n \ge 1$. Therefore from the Birkhoff Ergodic Theorem, $$\frac{\log J_{\Phi^n}(\hat\mu_\psi)(\omega, x)}{n} \mathop{\to}\limits_{n \to \infty} F_\Phi(\hat\mu_\psi),$$ for $\hat\mu_\psi$-almost all $(\omega, x) \in \Sigma_I^+\times \Lambda$.  
Moreover, from (\ref{jaco}) we have that 
\begin{equation}\label{ineqJ}
J_{\Phi^n}(\hat\mu_\psi)(\omega, x) \le C \cdot \frac{\mathop{\sum}\limits_{\Phi^n(\eta, y) = \Phi^n(\omega, x)} e^{S_n\psi(\eta, y)}}{e^{S_n\psi(\omega, x)}} \le C |I|^n \cdot e^{n(C_1 - C_2)},
\end{equation}
 for all $n \ge 1$, where $C_2 \le \psi \le C_1$ on $\Sigma_I^+\times \Lambda$ (as the potential $\psi$ is continuous). This implies that the sequence $\{\frac 1n \log J_{\Phi^n}(\hat\mu_\psi)(\omega, x)\}_n$ is bounded by $\log C+\log |I| + C_1 - C_1$, independently of $(\omega, x)$. 
Since $\log J_\Phi(\hat\mu_\psi)$ is integrable, we obtain then from the Birkhoff Ergodic Theorem, that $\int_{\Sigma_I^+\times \Lambda} \frac{\log J_{\Phi^n}(\hat\mu_\psi)(\omega, x)}{n} \ d\hat\mu_\psi(\omega, x) \mathop{\to}\limits_{n \to \infty} F_\Phi(\hat\mu_\psi)$, and similarly, $$\begin{aligned}
\gamma_n(\hat\mu_\psi, \delta):= &\int_{\Sigma_I^+\times \Lambda \setminus R_n(\hat\mu_\psi, \delta)} \big(\frac{\log J_{\Phi^n}(\hat\mu_\psi)}{n} - F_\Phi(\hat\mu_\psi)\big) \ d\hat\mu_\psi(\omega, x) = \\ 
&=\int_{\Sigma_I^+\times \Lambda} \big(\frac{\log J_{\Phi^n}(\hat\mu_\psi)}{n} - F_\Phi(\hat\mu_\psi)\big) \cdot \chi_{\Sigma_I^+\times \Lambda \setminus R_n(\hat\mu_\psi, \delta)} d\hat\mu_\psi(\omega, x) \mathop{\to}\limits_{n \to \infty} 0
\end{aligned}$$
Hence for any integer $n \ge 1$, $$\int_{\Sigma_I^+\times \Lambda} \frac{\log J_{\Phi^n}(\hat\mu_\psi)}{n} d\hat\mu_\psi = \gamma_n(\hat\mu_\psi, \delta)+ F_\Phi(\hat\mu_\psi)\cdot \hat\mu_\psi(\Sigma_I^+\times \Lambda \setminus R_n(\hat\mu_\psi, \delta))$$
Therefore, we obtain from (\ref{star}) that:
$$\begin{aligned} \int_{\Sigma_I^+\times \Lambda} & \frac{\log b_{k_n}((\omega, x), \tau, \hat\mu_\psi)}{k_n} d\hat\mu_\psi(\omega, x) < \hat\mu_\psi(R_{k_n}(\hat\mu_\psi, \delta)) (F_\Phi(\hat\mu_\psi) - \delta-\frac{\log 2}{k_n}) + \gamma_{k_n}(\hat\mu_\psi, \delta) +\\ & +F_\Phi(\hat\mu_\psi)\cdot \hat\mu_\psi(\Sigma_I^+\times \Lambda \setminus R_{k_n}(\hat\mu_\psi, \delta)) = \gamma_{k_n}(\hat\mu_\psi, \delta) + F_\Phi(\hat\mu_\psi) - \hat\mu_\psi(R_{k_n}(\hat\mu_\psi, \delta)(\delta + \frac{\log 2}{k_n})
\end{aligned}
$$
However if $\hat\mu_\psi(R_{k_n}(\hat\mu_\psi, \delta)) > \beta$ for $n > n(\delta)$ (for some integer $n(\delta) \ge 1$), then  it follows from the above and from the fact that: $\gamma_n(\hat\mu_\psi, \delta) \to 0$, that $$\begin{aligned}
\int_{\Sigma_I^+\times \Lambda} \frac{\log b_{k_n}((\omega, x), \tau, \hat\mu_\psi)}{k_n} d\hat\mu_\psi(\omega, x) &< F_\Phi(\hat\mu_\psi) - \beta(\delta +\frac{\log 2}{k_n}) + \gamma_{k_n}(\hat\mu_\psi, \delta) < F_\Phi(\hat\mu_\psi)
\end{aligned}
$$
 But then, this would give contradiction with Theorem \ref{fold}. Hence, for $\delta>0$ fixed there exists a sequence of positive numbers $\alpha_n \mathop{\to}\limits_{n \to \infty} 0$, such that the set $R_n(\hat\mu_\psi, \delta)$ of points $(\omega, x) \in \Sigma_I^+\times \Lambda$ for which $b_n((\omega, x), \tau, \hat\mu_\psi) < \frac 12 e^{n(F_\Phi(\hat\mu_\psi) - \delta)}$, has $\hat\mu_\psi$-measure that satisfies:
$$\hat\mu_\psi(R_n(\hat\mu_\psi, \delta)) < \alpha_n, \ \text{for} \ n > n(\delta)$$
Let denote now the complement of the set $R_n(\hat\mu_\psi, \delta)$ in $\Sigma_I^+\times \Lambda$ by: $$Q_n(\hat\mu_\psi, \delta):=\Sigma_I^+ \times \Lambda \setminus R_n(\hat\mu_\psi, \delta)$$ 
From the $\Phi$-invariance of  $\hat\mu_\psi$ on $\Sigma_I^+\times \Lambda$, and from the definition of $Q_n(\hat\mu_\psi, \delta)$, we obtain that $$\hat\mu_\psi(\Phi^n(Q_n(\hat\mu_\psi, \delta)) > 1-\alpha_n, \ n \ge n(\delta)$$  And from the definition of the set $\Phi^{n}(Q_n(\hat\mu_\psi, \delta))$, it follows that for any for point $(\eta', y') \in \Phi^n(Q_n(\hat\mu_\psi, \delta))$, there exist at least $\frac 12 e^{n(F_\Phi(\hat\mu_\psi) - \delta)}$ indices $\underline{i}=(i_1, \ldots, i_n) \in I^n$, such that $y' \in \phi_{\underline{i}}(\Lambda)= \phi_{i_1}\circ \ldots \circ \phi_{i_n}(\Lambda)$.  
From above, the sequence $\hat\mu_\psi(R_n(\hat\mu_\psi, \delta))$ converges to 0, so there exists an increasing sequence of integers $m_n \to \infty$ such that: \  $\hat\mu_\psi(R_{m_1}(\hat\mu_\psi, \delta)) < \frac{1}{2}, \ \hat\mu_\psi(R_{m_2}(\hat\mu_\psi)) < \frac{1}{2^2}, \ldots, \hat\mu_\psi(R_{m_n}(\hat\mu_\psi, \delta)) < \frac{1}{2^n}, \ldots$. Employing the sequence $\{m_n\}_n$, define now the following measurable subsets of $\Lambda$, 
$$\Lambda_n(\hat\mu_\psi, \delta):= \pi_2\big(\mathop{\cap}\limits_{s \ge n} \Phi^{m_s}(Q_{m_s}(\hat\mu_\psi, \delta))\big),$$
where $\pi_2: \Sigma_I^+\times \Lambda \to \Lambda$ is the canonical  projection to the second coordinate.
Moreover,  denote the union of the Borel subsets in $\Lambda$ introduced above by, $$\Lambda(\hat\mu_\psi, \delta):= \mathop{\cup}\limits_{n \ge 1} \Lambda_n(\hat\mu_\psi, \delta) = \pi_2\big(\mathop{\cup}\limits_{n \ge 1}\mathop{\cap}\limits_{s\ge n} \Phi^{m_s}(Q_{m_s}(\hat\mu_\psi, \delta))\big)$$
Firstly,  notice that  from the definition of the sequence of integers  $\{m_n\}_{n\ge 1}$, we have $$\hat\mu_\psi\big(\mathop{\cap}\limits_{s\ge n} \Phi^{m_s}(Q_{m_s}(\hat\mu_\psi, \delta))\big) \ge 1-\mathop{\sum}\limits_{s\ge n} \hat\mu_\psi\big(\Sigma_I^+\times \Lambda \setminus \Phi^{m_s}(Q_{m_s}(\hat\mu_\psi, \delta))\big) \ge 1 - \mathop{\sum}\limits_{s\ge n} \frac{1}{2^s} = 1- \frac{1}{2^{n-1}}$$
Therefore by taking the union of these sets over all $n \ge 1$, recalling that $\mu_\psi = \pi_{2*}(\hat\mu_\psi)$, and  observing that $\mu_\psi(\Lambda(\hat\mu_\psi, \delta)) = \hat\mu_\psi\big(\pi_2^{-1}(\Lambda(\hat\mu_\psi, \delta))\big)\ge \hat\mu_\psi\big(\mathop{\cup}\limits_{n\ge 1} \mathop{\cap}\limits_{s\ge n} \Phi^{m_s}(Q_{m_s}(\hat\mu_\psi, \delta))\big)$, we obtain that 
\begin{equation}\label{me}
\hat\mu_\psi\big(\mathop{\cup}\limits_{n\ge 1} \mathop{\cap}\limits_{s\ge n} \Phi^{m_s}(Q_{m_s}(\hat\mu_\psi, \delta))\big) = 1, \ \text{hence} \ \ \mu_\psi(\Lambda(\hat\mu_\psi, \delta)) = 1
\end{equation}

We now investigate the influence of the number of roots on the Hausdorff dimension of the set $\Lambda(\hat\mu_\psi, \delta)$. Recall from above that,  for any $(\eta', y') \in \Phi^n(Q_n(\hat\mu_\psi, \delta))$, there exist at least $\frac 12 e^{n(F_\Phi(\hat\mu_\psi) - \delta)}$ indices $\underline{i}=(i_1, \ldots, i_n) \in I^n$, such that $y' \in \phi_{\underline{i}}(\Lambda)= \phi_{i_1}\circ \ldots \circ \phi_{i_n}(\Lambda)$.  
Hence the points in the projection $\pi_2(\Phi^n(Q_n(\hat\mu_\psi, \delta)))$ are covered at least $\frac 12 e^{n(F_\Phi(\hat\mu_\psi) - \delta)}$ times by images of $\Lambda$, through compositions of $n$ maps of type $\phi_i$. 
Now,  $\cl S$ satisfies the condition that there exists $\kappa \in (0, 1)$ such that $|\phi_i'| < \kappa$ on $\Lambda$.  
It follows that, for any indices $i_1, \ldots, i_n \in I$, \ $\text{diam}(\phi_{i_1} \circ \ldots \circ \phi_{i_n}(\Lambda)) \le \kappa^n$.
Thus, every point in $\pi_2(\Phi^n(Q_n(\hat\mu_\psi, \delta)))$ can be covered at least $\frac 12 e^{n(F_\Phi(\hat\mu_\psi) - \delta)}$ times with sets of diameter less than $\kappa^n$.  \
For $\alpha \ge 0$, let us denote now by $t(\alpha)$ the unique zero of the following pressure function with respect to the shift map $\sigma: \Sigma_I^+ \to \Sigma_I^+$, 
\begin{equation}\label{tt}
t \to P_\sigma(t|\phi_{\omega_1}'(\sigma \omega)| - \alpha)
\end{equation}
Take an arbitrary number $t > t(F_\Phi(\hat\mu_\psi) - \delta)$; we assume without loss of generality that $F_\Phi(\hat\mu_\psi) > 0$ and that $\delta$ is small enough, so that $\delta < F_\Phi(\hat\mu_\psi)$. Let define the pressure function  $$p_\delta(s):= P( s|\phi_{\omega_1}'(\sigma \omega)| - F_\Phi(\hat\mu_\psi) + \delta), \ s \in \mathbb R $$ From assumption above, $p_\delta(t) < 0$. So from the conformality of the contractions $\phi_i$, and by denoting in general $\phi_\eta:= \phi_{\eta_1} \circ \ldots \circ \phi_{\eta_m}$ for  $\eta = (\eta_1, \ldots, \eta_m) \in I^m, m \ge 1$, it follows that for $n$ large:
\begin{equation}\label{pressure}
\mathop{\sum}\limits_{|\omega| = n} |\phi_{\omega}'|^t e^{-n(F_\Phi(\hat\mu_\psi) - \delta)} \le e^{\frac{n\cdot p_\delta(t)}{4} }
\end{equation}
Now for any $s \ge n$, from the above definition of $Q_{m_s}(\hat\mu_\psi, \delta)$, it follows that any point in $\Lambda_n(\hat\mu_\psi, \delta)$ can be covered with at least $M_s:= \frac 12 e^{m_s(F_\Phi(\hat\mu_\psi) - \delta)}$  sets $\phi_\eta(V)$ for $|\eta| = m_s$, and every one of these sets $\phi_\eta(V)$ has diameter less than $\kappa^{m_s}$. 
Denote the collection of the above sets $\phi_\eta(V)$ by $\cl U_s$, so $\cl U_s$ is a cover of $\Lambda_n(\hat\mu_\psi, \delta)$. \
We want now to perform extractions from this cover $\cl U_s$ of $\Lambda_n(\hat\mu_\psi, \delta)$ (by using its large multiplicity),  in such a way that in the end we obtain a subcover which is minimal, from the point of view of the sum of diameters raised to power $t$. This will be the  subcover which we shall use to estimate the Hausdorff dimension of the set $\Lambda_n(\hat\mu_\psi, \delta)$. \ 
We have that the maps $\phi_\eta$ are conformal, so we can apply the $5r$-Covering Theorem (see \cite{Mat}), where we consider $5U$ to denote the ball with the same center as $U$ and 5 times the radius of $U$. One can then extract a subfamily $\cl U_s(1) \subset \cl U_s$, such that the sets $5U, U \in \cl U_s(1)$,  cover $\Lambda_n(\hat \mu_\psi, \delta) $, and so that the sets in $\cl U_s(1)$ are mutually disjoint. From conformality we have that there exists $x, r$ and a fixed constant $C$ independent of $U$, such that $B(x, r) \subset U \subset B(x, Cr)$. \ 
We then eliminate this subfamily $\cl U_s(1)$. Since it was disjointed, \, the multiplicity of the cover $\cl U_s$ of $\Lambda_n(\hat\mu_\psi, \delta)$ is still at least $M_s -1$.
 Therefore we can repeat this procedure and will extract a second subfamily $\cl U_s(2)$ in $\cl U_s \setminus \cl U_s(1)$, which is disjointed and such that $5U, U \in \cl U_s(2)$ cover the set $\Lambda_n(\hat\mu_\psi, \delta)$. After eliminating both $\cl U_s(1)$ and $\cl U_s(2)$ from $\cl U_s$, the multiplicity of the cover is at least $M_s - 2$. 
By induction, we obtain thus $M_s$ subfamilies  $\cl U_s(j)$, which are disjointed and such that $5U, U \in \cl U_s(j)$, cover $\Lambda_n(\hat\mu_\psi, \delta)$. 
We then take, out of these subfamilies constructed above,  the subfamily $\cl U_s(j_0)$ for which the expression $\mathop{\sum}\limits_{U \in \cl U_s(j_0)} (\text{diam} U)^t$ is minimal.  Then from (\ref{pressure}), we obtain:
\begin{equation}\label{min}
\mathop{\sum}\limits_{U \in \cl U_s(j_0)} (\text{diam} U)^t \le \frac{1}{M_s} \mathop{\sum}\limits_{U \in \cl U_s} (\text{diam} U)^t \le C e^{m_sp_\delta(t)/4} < 1,
\end{equation}
for some constant $C>0$,  independent of $s, n$ large. Since for any $s \ge n$, we can obtain such minimal covers $\cl U_s(j_0)$ for the set $\Lambda_n(\hat\mu_\psi, \delta)$ , and since $t$ was chosen arbitrarily larger than $t(F_\Phi(\hat\mu_\psi)- \delta)$, it follows from (\ref{min}) that: $$HD(\Lambda_n(\hat\mu_\psi, \delta)) \le t(F_\Phi(\hat\mu_\psi) - \delta)$$
Now recall the definition of $\Lambda(\hat\mu_\psi, \delta) = \mathop{\cup}\limits_{n \ge 1} \Lambda_n(\hat\mu_\psi, \delta)$. From the last estimate, we infer that $$HD(\Lambda(\hat\mu_\psi, \delta)) \le t(F_\Phi(\hat\mu_\psi) - \delta)$$
Also from (\ref{me}), $\mu_\psi(\Lambda(\hat\mu_\psi, \delta)) =1$. Define now the set $\Lambda(\psi):= \mathop{\cap}\limits_{\delta >0} \Lambda(\hat\mu_\psi, \delta) = \mathop{\cap}\limits_{n \ge 1} \Lambda(\hat\mu_\psi, \frac 1n)$. We have then that $\mu_\psi(\Lambda(\psi)) = 1$. Let us now remark that from definition (\ref{tt}) of the zero $t(\alpha)$, and from the continuity of the pressure function, we obtain that  $t(F_\Phi(\hat\mu_\psi) - \delta) \to t(F_\Phi(\hat\mu_\psi))$ when $\delta \to 0$. But from Theorem \ref{fold},  we know that $\log o(\cl S, \psi) = F_\Phi(\hat\mu_\psi)$. Hence, by taking the set $Z(\psi):= \Lambda \setminus \Lambda(\psi),$ we have $\mu_\psi(Z(\psi)) = 0$; thus from the definition of $HD(\mu_\psi)$,  \ $HD(\mu_\psi) \le HD(\Lambda \setminus Z(\psi)) \le t(\cl S, \psi)$.

\end{proof}

\section{Applications to Bernoulli convolutions.}

Consider the random series $\mathop{\sum}\limits_{n \ge 0} \pm \lambda^n$ for $\lambda \in (0, 1)$ where the $+, -$ signs are taken independently and with equal probability, and let us denote its distribution by $\nu_\lambda$. This is called a Bernoulli convolution, since it is in fact the infinite convolution of the atomic measures $\frac 12(\delta_{-\lambda^n}+ \delta_{\lambda^n})$, for $n \ge 0$ (for eg \cite{E}, \cite{So}). The probability measure $\nu_\lambda$ can be written also as the self-similar measure associated to the probability vector $(\frac 12, \frac 12)$ and to the  iterated function system $$\cl S_\lambda = \{S_{1}, S_{2}\},$$
where $S_{1}(x) = \lambda x-1, \ S_{2}(x) = \lambda x+1, \ x \in \mathbb R$. Hence, $\nu_\lambda$ satisfies the self-similarity relation: $$\nu_\lambda = \frac 12 \nu_\lambda \circ S_1^{-1} + \frac 12 \nu_\lambda \circ S_2^{-1}$$ 
The case $\lambda \in (0, \frac 12)$ corresponds to $\cl S_\lambda$ having no overlaps, while the case when $\lambda \in [\frac 12, 1)$ corresponds to the more difficult situation of the iterated function system  $\cl S_\lambda$ having overlaps. 
We assume in the sequel that $\lambda \in (\frac 12, 1)$, thus we are in the case when $\cl S_\lambda$ has overlaps. 
The associated limit set $\Lambda_\lambda$ is in this case the whole interval $I_\lambda = [-\frac{1}{1-\lambda}, \frac{1}{1-\lambda}]$.
 The measure $\nu_\lambda$ can be viewed also as the projection $\pi_{\lambda*}\nu_{(\frac 12, \frac 12)}$, where $\nu_{(\frac 12, \frac 12)}$ is the Bernoulli measure on $\Sigma_2^+$ generated by the vector $(\frac 12, \frac 12)$, and $\pi_\lambda: \Sigma_2^+ \to I_\lambda$ is the canonical coding map. It is well-known that the measure $\nu_\lambda$ can be either singular or absolutely continuous. Several results on Bernoulli convolutions are in the paper by Peres, Schlag and Solomyak \cite{PSS}.
 The case $\lambda > \frac 12$ attracted a lot of interest, starting with Erd\"os who proved in \cite{E} that,  when $\frac 1\lambda$ is a Pisot number, then $\nu_\lambda$ is singular. Then later Solomyak showed in \cite{So} that the measure $\nu_\lambda$ is absolutely continuous for Lebesgue-a.e $\lambda \in [\frac 12, 1)$; the method of transversality was used in \cite{So}, and also by Peres and Schlag \cite{PSch}, and Peres and Solomyak \cite{PS}.  Notice that, if $\nu_\lambda$ is absolutely continuous, then $HD(\nu_\lambda) = 1$. From the point of view of actual values of $\lambda$ involved,  Garsia proved in \cite{G} that $\nu_\lambda$ is absolutely continuous when $\lambda^{-1}$ is an algebraic integer in $(1, 2)$, whose monic polynomial has other roots outside the unit circle and constant coefficient $\pm 2$. For example when $\lambda^{-1} = 2^{\frac 1m}, \ m \ge 2$, $\nu_\lambda$ is absolutely continuous,  so $HD(\nu_\lambda) = 1$.  
Przytycki and Urba\'nski (\cite{PU}) proved that, if $\lambda^{-1}$ is the inverse of a Pisot number in $(1, 2)$, then $HD(\nu_\lambda) < 1$.  
In the special case when $\lambda = \frac{\sqrt 5 - 1}{2}$ (the reciprocal of the Pisot number $\frac{\sqrt{5} +1}{2}$, the golden mean), Alexander and Zagier found in \cite{AZ}  precise estimates for $HD(\nu_\lambda)$, and they showed that $0.99557 < HD(\nu_\lambda) < 0.99574$. Recently, Hochman showed in \cite{H} that $HD(\nu_\lambda) = 1$ for $\lambda$ outside a parameter  set of dimension zero in $(\frac 12, 1)$.
 
 For arbitrary $\lambda \in (\frac 12, 1)$,  Theorem \ref{dimest} below  gives an upper estimate for $HD(\nu_\lambda)$, by using an expression involving $o(\cl S_\lambda)$; this allows to obtain bounds also for the overlap numbers $o(\cl S_\lambda)$. 
In particular, if  $HD(\nu_\lambda) = 1$ for some value $\lambda \in (\frac 12, 1)$, then $o(\cl S_\lambda) \le 2\lambda$.
In general,  $1 \le o(\cl S_\lambda) \le 2$, for any $\lambda \in (\frac 12, 1)$; we show that in fact, the overlap number $o(\cl S_\lambda)$ is \textit{never} equal to 2 (even if, for $\lambda \to 1$ the overlaps become larger). For specific values of $\lambda$ (for eg $\lambda = 2^{-\frac 1m}, m \ge 2$, or $\lambda = \frac{\sqrt 5 -1}{2}$), we obtain then more precise bounds for $o(\cl S_\lambda)$. \ 
First, for arbitrary $\lambda \in (\frac 12, 1)$,  the measure $\nu_\lambda$ is supported on the limit set of $\cl S_\lambda$, which is the interval  $I_\lambda = [-\frac{1}{1-\lambda}, \frac{1}{1-\lambda}]$; the coding map is $\pi_\lambda: \Sigma_2^+ \to I_\lambda$. Recall that for $x \in  I_\lambda$ and $n\ge 2$, \ $\beta_n(x)$ denotes the number of $n$-chains $(\zeta_1, \ldots, \zeta_n) \in \{1, 2\}^n$ from points in $I_\lambda$ to $x$, i.e.  $x \in \phi_{\zeta_1\ldots \zeta_n}\big([-\frac{1}{1-\lambda}, \frac{1}{1-\lambda}]\big)$. From Corollary \ref{usorcalc}, in the formula for $o(\cl S_\lambda)$ we integrate $\log \beta_n$ with respect to the uniform Bernoulli measure $\nu_{(\frac 12, \frac 12)}$.
 
\begin{thm}\label{dimest}
For all $\lambda \in (\frac 12, 1)$, the following relation is satisfied for the Bernoulli convolution $\nu_\lambda$:
$$HD(\nu_\lambda) \  \le \ \frac{\log\frac{2}{o(\cl S_\lambda)}}{|\log \lambda|},$$
where $o(\cl S_\lambda)$ denotes the overlap number of $\cl S_\lambda$, which can be computed as: $$o(\cl S_\lambda) =  \exp\Big(\mathop{\lim}\limits_{n \to \infty} \frac 1n \int_{\Sigma_2^+} \log \beta_n(\pi_\lambda \omega) \ d\nu_{(\frac 12, \frac 12)}(\omega)\Big)$$ And from the above,   $o(\cl S_\lambda) \le 2 \lambda^{HD(\nu_\lambda)}$.
\end{thm}

\begin{proof}

From Theorem \ref{equal}, in our case the measure $\nu_\lambda$ can be written as $\pi_{\lambda*}\nu_{(\frac 12, \frac 12)}$ and it is equal to the $\pi_2$-projection of an equilibrium state $\hat \mu_\psi$ on $\Sigma_2^+ \times I_\lambda$. Therefore, from Corollary \ref{usorcalc}, $$o(\cl S_\lambda) = \exp\big(\mathop{\lim}\limits_{n \to \infty} \frac 1n \int_{\Sigma_2^+} \log \beta_n(\pi_\lambda \omega) \ d\nu_{(\frac 12, \frac 12)}(\omega)\big)$$
$\cl S_\lambda$ is a system of similarities, thus from Theorem \ref{dim}, $HD(\nu_\lambda)$ is bounded above by the unique zero of the  pressure function with respect to $\sigma: \Sigma_2^+ \to \Sigma_2^+$: $$t \to P_\sigma(t \log \lambda - o(\cl S_\lambda)) = t \log \lambda + \log 2 - \log o(\cl S_\lambda)$$ Hence it follows that $HD(\nu_\lambda) \le \ \frac{\log \frac{2}{o(\cl S_\lambda)}}{|\log \lambda|}$, and  the corresponding bound for $o(\cl S_\lambda)$.

\end{proof}

For any $\lambda \in (\frac 12, 1)$,  the number of overlaps between images $S_{i_1\ldots i_n}(I_\lambda)$ is less than $2^n$, so $1 \le o(\cl S_\lambda) \le 2$. In fact, it turns out that the overlap number of $\cl S_\lambda$ is always strictly less than 2:

\begin{cor}\label{2}
In the above setting, it follows that for all parameters $\lambda \in (\frac 12, 1)$,  $$o(\cl S_\lambda) < 2$$

\end{cor}

\begin{proof}

If $o(\cl S_\lambda) = 2$, then from Theorem \ref{dimest}, it would follow that $\lambda = 1$. Hence contradiction.
\end{proof}

For a large set of values of $\lambda$, by using Theorem \ref{dimest} and the above mentioned results of \cite{AZ}, \cite{G}, \cite{H}, \cite{So}, we can obtain more precise estimates for the overlap number:

\begin{cor}\label{numerical}
 a) For $\lambda$ outside a set of dimension zero in $(\frac 12, 1)$, we have $$o(\cl S_\lambda) \le 2\lambda$$
This happens for example when $\lambda^{-1}$ is an algebraic number whose monic polynomial has other roots outside the unit circle and constant coefficient $\pm 2$. In particular, if $\lambda = 2^{-\frac 1m}$ for $m \ge 2$,  then $$o(\cl S_\lambda) \le 2^{\frac {m-1}{m}} $$

b) In case $\lambda = \frac{\sqrt 5 -1}{2}$, then \ $o(\cl S_\lambda) \le 2 \lambda^{0.99557} < 1.25$.

\end{cor}

\

Let now $p$ arbitrary in $(0, 1)$ and denote by  $\nu_{(p, 1-p)}$ the Bernoulli measure on $\Sigma_2^+$ determined by the vector $(p, 1-p)$.
For $\lambda \in (\frac 12, 1)$, one defines the \textit{biased} Bernoulli convolution $\nu_{\lambda, p}$ (see for eg \cite{PS}), where $\nu_{\lambda, p}$ is the $\pi_\lambda$-projection of $\nu_{(p, 1-p)}$  onto the limit set $I_\lambda = [-\frac {1}{1-\lambda}, \frac {1}{1-\lambda}]$. We have as above  the associated lift map $\Phi_\lambda: \Sigma_2^+\times I_\lambda \to \Sigma_2^+\times I_\lambda$. From the discussion before Theorem \ref{equal}, there exists a $\Phi_\lambda$-invariant equilibrium measure $\hat\nu_{\lambda, p}$ on $\Sigma_2^+\times I_\lambda$, such that $\pi_{2*}\hat\nu_{\lambda, p} = \nu_{\lambda, p}$. For  integers $0 < k < n$, denote by $W(x, n, k)$ the set of $n$-chains $(i_1, \ldots, i_n) \in \{1, 2\}^n$ from points in $I_\lambda$ to $x$, having exactly $k$ indices $i_j$ equal to 1. From (\ref{betap}), for any $x\in I_\lambda$, $\tau >0$ and $n \ge 2$, we have
$$\beta_n\big(x, \tau |\log\frac{p}{1-p}|, (p, 1-p)\big) = \mathop{\sum}\limits_{k, \ |\frac{k}{n} - p| < \tau}\text{Card} \ W(x, n, k)$$
Thus, for any parameter $\lambda \in (\frac 12, 1)$, it follows  from Theorem \ref{dim} and Corollary \ref{usorcalc}  that:

\begin{cor}\label{biased}
For all $\lambda \in (\frac 12, 1)$ and $p \in (0, 1)$, the biased Bernoulli convolution $\nu_{\lambda, p}$ satisfies:
$$HD(\nu_{\lambda, p}) \  \le \ \frac{\log\frac{2}{o(\cl S_{\lambda}, \hat\nu_{\lambda, p})}}{|\log \lambda|},$$
where $o(\cl S_\lambda, \hat\nu_{\lambda, p})$ denotes the overlap number of $\cl S_\lambda$ with respect to $\hat\nu_{\lambda, p}$, which can be computed by: $$o(\cl S_\lambda, \hat\nu_{\lambda, p}) =  \exp\Big(\mathop{\lim}\limits_{\tau \to 0} \mathop{\lim}\limits_{n \to \infty} \frac 1n  \int_{\Sigma_2^+} \log \mathop{\sum}\limits_{|\frac{k}{n} - p| < \tau} \text{Card} \ W(\pi_\lambda\omega, n, k) \ d\nu_{(p, 1-p)}(\omega)\Big) $$
\end{cor}

\

\textbf{Acknowledgements:} The second-named author was supported in part by the NSF Grant DMS 1361677.

\

Eugen Mihailescu, \ \  Eugen.Mihailescu\@@imar.ro

Institute of Mathematics of the Romanian Academy, P. O. Box 1-764,
RO 014700, 

Bucharest, Romania.

 Webpage: www.imar.ro/$\sim$mihailes

\

Mariusz Urba\'nski, \ \ urbanski@unt.edu

Department of Mathematics, University of North Texas,
  Denton, TX 76203-1430, USA.

Webpage: www.math.unt.edu/$\sim$urbanski

\end{document}